\documentclass[twoside,11pt]{amsart}

\usepackage{a4wide}
\usepackage{latexsym}
\usepackage{mathrsfs}
\usepackage[T1]{fontenc}
\usepackage{enumitem}
\usepackage{mathrsfs}

\usepackage{amssymb}
\usepackage{latexsym}
\usepackage{amsmath}
\usepackage{fancyhdr}
\usepackage{bbm}
\usepackage{setspace}
\usepackage{mathabx}

\numberwithin{equation}{section}

\newtheorem{theorem}{Theorem}[section] 
\newtheorem{lemma}[theorem]{Lemma}     
\newtheorem{corollary}[theorem]{Corollary}
\newtheorem{proposition}[theorem]{Proposition}
\theoremstyle{definition}

\newtheorem{definition}[theorem]{Definition}

\newcommand {\C}{\mathbb C}

\newcommand {\N}{\mathbb N}

\newcommand {\K}{\mathscr K}
\newcommand {\B}{\mathscr B}

\newcommand{\spn}{{\rm span \,}}

\newcommand{\im}{\operatorname{im}}

\title[Subspaces as Kernels of Bounded Operators]{Subspaces that can and cannot be the kernel of a bounded operator on a Banach space}

\begin{document}

\author[N.\ J.\  Laustsen]{Niels Jakob Laustsen}
\author[J.\ T.\  White]{Jared T.\ White}
\address{
Niels Jakob Laustsen\\
Department of  Mathematics and Statistics\\
Lancaster University\\
Lancaster LA1 4YF\\
United Kingdom.}
\email{n.laustsen@lancaster.ac.uk}
\address{
Jared T. White\\
Laboratoire de Math{\'e}matiques de Besan{\c c}on\\
Universit{\'e} de Franche-Comt{\'e}\\
16 Route de Gray\\
25030 Besan{\c c}on\\
France.}
\email{jared.white@univ-fcomte.fr, jw65537@gmail.com}

\date{2018}


\begin{abstract}
Given a Banach space $E$, we ask which closed subspaces may be realised as the kernel of a bounded operator $E \rightarrow E$. We prove some positive results which imply in particular that when $E$ is separable every closed subspace is a kernel. Moreover, we show that there exists a Banach space $E$ which contains a closed subspace that cannot be realised as the kernel of any bounded operator on $E$. This implies that the Banach algebra $\mathcal{B}(E)$ of bounded operators on $E$ fails to be weak*-topologically left Noetherian in the sense of \cite{Wh}. The Banach space $E$ that we use is the dual of one of Wark's non-separable, reflexive Banach spaces with few operators.
\end{abstract}

\maketitle

\section{Introduction}
\noindent
In this note we address the following natural question: given a Banach space $E$, which of its closed linear subspaces $F$ are the kernel of some bounded linear operator $E \rightarrow E$?  We shall begin by showing that if either $E/F$ is separable, or $F$ is separable and  $E$ has the separable complementation property, then $F$ is indeed the kernel of some bounded operator on $E$ (Propositions \ref{2.0} and \ref{2.0a}). 
Our main result is that there exists a reflexive, non-separable Banach space $E$ for which these are the only closed linear subspaces that may be realised as kernels (Theorem \ref{2.1}), and in particular $E$ has a closed linear subspace that cannot be realised as the kernel of a bounded linear operator on $E$ (Corollary \ref{2.3}). The Banach space in question may be taken to be the dual of any reflexive, non-separable Banach space that has few operators, in the sense that every bounded operator on $E$ is the sum of a scalar multiple of the identity and an operator with separable range. Wark has shown that such Banach spaces exist \cite{Wa, Wa2018}.

We now describe how we came to consider this question. Given a Banach space $E$ we write $E'$ for its dual space, and $\B(E)$ for the algebra of bounded linear operators $E \rightarrow E$. We recall that a \emph{dual Banach algebra} is a Banach algebra $A$ which is isomorphically a dual Banach space in such a way that the multiplication on $A$ is separately weak*-continuous; equivalently $A$ has a predual which may be identified with a closed $A$-submodule of $A'$.  When $E$ is a reflexive Banach space, $\B(E)$ is a dual Banach algebra with predual given by $E \widehat{\otimes} E',$ where $\widehat{\otimes}$ denotes the projective tensor product of Banach spaces. We recall the following definition from \cite{Wh}:

\begin{definition}
Let $A$ be a dual Banach algebra. We say that $A$ is \textit{weak*-topologically left Noetherian} if every weak*-closed left ideal $I$ of $A$ is weak*-topologically finitely-generated, i.e. there exists $n \in \N$, and there exist $x_1, \ldots, x_n \in I$ such that 
$$I = \overline{Ax_1 +\C x_1 + \cdots +Ax_n +\C x_n}^{w^*}.$$
\end{definition}

In \cite{Wh} various examples were given of dual Banach algebras which satisfy this condition, but none were found that fail it. Using our main result, we are able to prove in Theorem \ref{2.5.13} of this note that, for any non-separable, reflexive Banach space $E$ with few operators in the above sense, $\B(E')$ is a dual Banach algebra which is not weak*-topologically left Noetherian.

\section{Results}
\noindent
We first show that in many cases closed linear subspaces can be realised as kernels. In particular, for a separable Banach space every closed linear subspace is the kernel of a bounded linear operator. Given a Banach space $E$, and elements $x \in E, \lambda \in E'$, we use the bra-ket notation $\vert x \rangle \langle \lambda \vert$ 
to denote the rank-one operator $y \mapsto \langle y, \lambda \rangle x$.

\begin{proposition}		\label{2.0}
Let $E$ be a Banach space, and let $F$ be a closed linear subspace of $E$ such that $E/F$ is separable. Then there exists $T \in \B(E)$ such that $\ker T = F$.
\end{proposition}

\begin{proof}
Since $E/F$ is separable, the unit ball of $F^\perp \cong (E/F)'$ is weak*-metrisable, and hence, since it is also compact, it is separable. Therefore we may choose a sequence of functionals $(\lambda_n)$ which is weak*-dense in the unit ball of $F^\perp$. We may assume that $E$ is infinite dimensional, since otherwise the result follows from elementary linear algebra. We may therefore pick a normalised basic sequence $(b_n)$ in $E$. Define $T \in \B(E)$ by
$$T = \sum_{n=1}^\infty 2^{-n} \vert  b_n \rangle \langle \lambda_n \vert.$$
Since each $\lambda_n$ belongs to $F^\perp$, clearly $F \subset \ker T$. Conversely, if $x \in \ker T$ then, since $(b_n)$ is a basic sequence, we must have $\langle x, \lambda_n \rangle = 0$ for all $n \in \N$. Hence 
$$x \in \{ \lambda_n : n \in \N \}_\perp = \left( \overline{ \operatorname{span}}^{w^*} \{ \lambda_n : n \in \N \}\right)_\perp = (F^\perp)_\perp = F,$$
as required.
\end{proof}

A Banach space~$E$ is said to have the 
\emph{separable complementation property} if, for each separable linear subspace~$F$ of~$E$,
there is a separable, complemented linear subspace~$D$ of~$E$ such that $F \subset D$. For such Banach spaces we can show that every separable closed linear subspace is a kernel. By \cite{L} every reflexive Banach space has the separable complementation property, so that the next proposition applies in particular to the duals of Wark's Banach spaces, which we shall use in our main theorems. We refer to \cite{HMVZ} for a survey of more general classes of Banach spaces that enjoy the separable complementation
property.

\begin{proposition}		\label{2.0a}
Let~$E$ be  a Banach space with the separable
complementation property. Then, for every closed, separable linear
subspace~$F$ of~$E$, there exists $T\in \B(E)$ such that $\ker T =
F$.
\end{proposition}
\begin{proof}
Choose a separable, complemented linear subspace~$D$ of~$E$ such that
$F \subset D$, and let  $P \in \B(E)$ be a projection with
range~$D$. By Proposition \ref{2.0}, we can
find $S \in \B(D)$ such that $\ker S= F$. Then
$$ T\colon\ x\mapsto SPx + x-Px,\quad E\to E, $$
defines a bounded linear operator on~$E$. We shall now complete the
proof by showing that $\ker T = F$.

Indeed, for each $x \in \ker T$ we have $0 = (\operatorname{id}_E-P)Tx
= x- Px$, so that $Px = x$. This implies that $0=Tx = Sx$, and
therefore $x\in F$. Conversely, each $x\in F$ satisfies $Sx=0$ and $Px
= x$, from which it follows that $Tx=0$.  
\end{proof}

We recall some notions from Banach space theory that we shall require. Let $E$ be a Banach space. A \textit{biorthogonal system} in $E$ is a set 
$$\{ (x_\gamma, \lambda_\gamma) : \gamma \in \Gamma \} \subset E \times E',$$
 for some indexing set $\Gamma$,
with the property that
\begin{align*}
\langle x_\alpha, \lambda_\beta \rangle = 
	\begin{cases} 
	1\ &\text{if}\ \alpha=\beta\\ 
	0\ &\text{otherwise}\end{cases}
	\quad (\alpha, \beta \in \Gamma).
\end{align*}
A \textit{Markushevich basis} for a Banach space $E$ is a biorthogonal system $\{ (x_\gamma, \lambda_\gamma) : \gamma \in \Gamma \}$ in $E$ such that $\{\lambda_\gamma : \gamma \in \Gamma \}$ separates the points of $E$ and such that $\overline{\spn} \{x_\gamma : \gamma \in \Gamma \} = E$. For an in-depth discussion of Markushevich  bases see \cite{HMVZ}, in which a Markushevich  basis is referred to as an ``M-basis''.

We now prove a lemma and its corollary which we shall use to prove Corollary \ref{2.3} below.

\begin{lemma}		\label{2.5.8}
Let $E$ be a Banach space containing an uncountable bi\-o\-r\-t\-h\-ogonal system. Then $E$ contains a closed linear subspace $F$ such that both $F$ and $E/F$ are non-separable.
\end{lemma}

\begin{proof}
Let $\left\{ (x_\gamma, \lambda_\gamma) : \gamma \in \Gamma \right\}$ be an uncountable biorthogonal system in $E$.
We can write $\Gamma = \bigcup_{n=1}^\infty \Gamma_n$, where
$$\Gamma_n = \{ \gamma \in \Gamma : \Vert x_\gamma \Vert, \Vert \lambda_\gamma \Vert \leq n \} \quad (n \in \N).$$
Since $\Gamma$ is uncountable, there must exist an $n \in \N$ such that $\Gamma_n$ is uncountable.
 Let $\Delta$ be an uncountable subset of $\Gamma_n$ such that $\Gamma_n \setminus \Delta$ is also uncountable, and set $F = \overline{\spn} \{x_\gamma : \gamma \in \Delta \}$. The subspace $F$ is non-separable since $\{x_\gamma : \gamma \in \Delta \}$ is an uncountable set satisfying 
$$\Vert x_\alpha - x_\beta \Vert \geq \frac{1}{n} \vert \langle x_\alpha - x_\beta, \lambda_\alpha \rangle \vert = \frac{1}{n} \quad 
(\alpha, \beta \in \Delta, \alpha \neq \beta).$$

Let $q \colon E \rightarrow E/F$ denote the quotient map. It is well known that the dual map $q' \colon (E/F)' \rightarrow E'$ is an isometry with image equal to $F^\perp$. For each $\gamma \in \Gamma_n \setminus \Delta$ the functional $\lambda_\gamma$ clearly belongs to $F^\perp$, so that there exists $g_\gamma \in (E/F)'$ such that $q'(g_\gamma) = \lambda_\gamma$, and such that 
$\Vert g_\gamma \Vert = \Vert \lambda_\gamma \Vert$. We now see that $\{ q(x_\gamma) : \gamma \in \Gamma_n \setminus \Delta \}$ is an uncountable 
$1/n$-separated subset of $E/F$ because
\begin{align*}
\Vert q(x_\alpha) - q(x_\beta) \Vert &\geq \frac{1}{n} \vert \langle q(x_\alpha) - q(x_\beta), g_\alpha \rangle \vert
=\frac{1}{n} \vert \langle x_\alpha - x_\beta, q'(g_\alpha) \rangle \vert \\
&= \frac{1}{n} \vert \langle x_\alpha - x_\beta, \lambda_\alpha \rangle \vert = \frac{1}{n} \qquad (\alpha, \beta \in \Gamma_n \setminus \Delta, \ \alpha \neq \beta).
\end{align*}
It follows that $E/F$ is non-separable.
\end{proof}

\begin{corollary}		\label{2.5.7}
Let $E$ be a non-separable, reflexive Banach space. Then $E$ contains a closed linear subspace $F$ such that both $F$ and $E/F$ are non-separable.
\end{corollary}

\begin{proof}
By \cite[Theorem 5.1]{HMVZ} $E$ has a Markushevich basis 
$\left\{ (x_\gamma, \lambda_\gamma) : \gamma \in \Gamma \right\}$. 
The set $\Gamma$ must be uncountable since $E$ is non-separable and, by the definition of a Markushevich basis, $\overline{\spn}\{x_\gamma : \gamma \in \Gamma \} = E$. Hence the result follows from Lemma \ref{2.5.8}.
\end{proof}

We now move on to discuss our main example.
Building on the work of Shelah and Stepr{\=a}ns~\cite{SS}, Wark constructed in \cite{Wa} a reflexive Banach space $E_W$ with the property that it is non-separable but has few operators in the sense that 
\begin{equation}	\label{eq2.5.2} 
\B(E_W) = \C \, {\rm id}_{E_W} + \mathscr{X}(E_W),
\end{equation}
 where $\mathscr{X}(E_W)$ denotes the ideal of operators on $E_W$ with separable range. Recently Wark gave a second example of such a space \cite{Wa2018} with the additional property that the space is uniformly convex. For the rest of our paper $E_W$ can be taken to be either of these spaces. In particular, the only properties of $E_W$ that we shall make use of are that it is reflexive, non-separable, and satisfies Equation \eqref{eq2.5.2}.
\vskip 2mm

\textit{Remark.}
We briefly outline why the dual Banach algebra $\B(E_W')$ fits into the framework of~\cite{Wh}. A \emph{transfinite basis} for a Banach space~$X$ is a linearly independent family 
$\{ x_\alpha : \alpha < \gamma \}$ of vectors in $X$, where $\gamma$ is some infinite ordinal, such that $X_0  = \operatorname{span}\{ x_\alpha : \alpha<\gamma\}$ is dense in~$X$, and with the property that there is a constant $C\geqslant 1$ such that, for each ordinal $\beta<\gamma$, the linear  map~$P_\beta\colon X_0\to X$ defined by
\[ P_\beta \Bigl( \sum_{\alpha<\gamma} s_\alpha x_\alpha \Bigr) = \sum_{\alpha<\beta} s_\alpha x_\alpha \]
has norm at most~$C$. 


In the notation of \cite{Wa} and \cite{Wa2018} the 
family $\{ e(\alpha) : \alpha < \omega_1 \}$ is a transfinite basis of $E_W$.
 See the proofs of Theorem 2 in
\cite{Wa} or Proposition 8 in \cite{Wa2018}. It is shown in \cite{Ros} 
that Banach spaces with transfinite bases have
the approximation property.
 Since the duals of reflexive Banach spaces with the approximation
property also have this property, $E_W'$ has the approximation property.
It follows from \cite[Corollary 5.6]{Wh} that the algebra of compact operators $\K(E_W')$ is a compliant Banach algebra in the sense of \cite[Definition 5.4]{Wh}. 
Hence $\K(E_W')$, and its multiplier algebra $\B(E_W')$, fit into the framework of that paper. In particular \cite[Theorem 6.3]{Wh} gives a complete description of the weak*-closed left ideals of $\B(E_W')$, although we shall not need this in the sequel.
\vskip 2mm

In what follows we denote the image of a bounded linear operator $T$ by $\im T$.

\begin{proposition}		\label{2.5.6}
Let $F$ be a closed linear subspace of the Banach space $E_W$ with the property that 
$F = \overline{\im T_1+ \cdots + \im T_n}$, for some $n \in \N$, and $T_1, \ldots, T_n \in \B(E_W)$. Then either $F$ or $E_W/F$ is separable.
\end{proposition}

\begin{proof}
Suppose that $F = \overline{ \im T_1+ \cdots + \im T_n},$ for some $n \in \N,$ and some $T_1, \ldots, T_n \in \B(E_W)$. 
By \eqref{eq2.5.2} there exist $\alpha_1, \ldots, \alpha_n \in \C$ and $S_1, \ldots, S_n \in \mathscr{X}(E_W)$ such that
$$T_i = \alpha_i {\rm id}_{E_W} + S_i \quad (i=1, \ldots, n).$$
If every $\alpha_i$ equals zero, then $F = \overline{\im S_1+ \cdots + \im S_n}$, which is separable.  Otherwise, without loss of generality, we may assume that $\alpha_1 \neq 0$. Let $x \in E_W$. Then $T_1x = \alpha_1 x + S_1x$, implying that
$$x = \frac{1}{\alpha_1} \left(T_1x - S_1x \right) \in F + \overline{\im S_1}.$$
As $x$ was arbitrary, it follows that $E_W = F + \overline{\im S_1}$, so that 
$$E_W/F = \frac{\left(F + \overline{\im S_1} \right)}{F} \cong \frac{\overline{\im S_1}}{\left(\overline{\im S_1} \cap F \right)}.$$
Hence, it follows that $E_W/F$ is separable.
\end{proof}

We can now prove our two theorems.

\begin{theorem}		\label{2.1}
Let $D$ be a closed linear subspace of $E_W'$. Then the following conditions are equivalent:
	\begin{enumerate}
		\item[{\rm (a)}] either $D$ or $E_W'/D$ is separable;
		\item[{\rm (b)}] $D = \ker T$, for some $T \in \B(E_W')$;
		\item[{\rm (c)}] there exist $n \in \N$, and $T_1, \ldots, T_n \in \B(E_W')$ such that 
		                 $D = \bigcap_{i = 1}^n \ker T_i$.
	\end{enumerate}
\end{theorem}

\begin{proof}
We first prove that (a) implies (b). Indeed, let $D$ be a closed linear subspace of $E_W'$. By \cite{L}, we may apply Proposition \ref{2.0a} to $E_W'$ to see that if $D$ is separable then it may be realised as the kernel of some $T \in \B(E_W')$. If instead $E_W'/D$ is separable, then we may apply Proposition \ref{2.0}.
	
It is trivial that (b) implies (c), so it remains to prove that (c) implies (a). Let $D$ be a closed linear subspace of $E_W'$ that can be written in the given form for some $n \in \N$ and some $T_1, \ldots, T_n \in \B(E_W')$. Set $F = D_\perp$. Since $E_W$ is reflexive, 
there exist $S_1, \ldots, S_n \in \B(E_W)$ such that, for each $i = 1, \ldots, n$, $T_i = S_i'$, the dual operator of $S_i$. It follows that
$$F = D_\perp = \left( \bigcap_{i = 1}^n \ker S_i' \right)_\perp = 
\overline{\im S_1+ \cdots + \im S_n}.$$
Hence, by Proposition \ref{2.5.6}, either $F$ or $E_W/F$ is separable. 
Since $F$ is also reflexive, it now follows from the formulae $(E_W/F)' \cong D$ and $F' \cong E_W'/D$ that either $D$ or $E_W'/D$ is separable.
\end{proof}

\begin{corollary}		\label{2.3}
The Banach space $E_W'$ contains a closed linear subspace $D$ which is not of the form $\bigcap_{i=1}^n \ker T_i$, for any choice of $n \in \N$, and operators $T_1, \ldots, T_n \in \B(E_W')$.
\end{corollary}

\begin{proof}
By Corollary \ref{2.5.7} $E_W'$ contains a closed linear subspace $D$ such that neither $D$ nor $E_W'/D$ is separable. The result now follows from Theorem \ref{2.1}.
\end{proof}

\begin{theorem} 	\label{2.5.13}
The dual Banach algebra $\B(E_W')$ is not weak*-topologically left Noetherian.
\end{theorem}

\begin{proof}
To simplify notation set $E= E_W'$. Let $D$ be a closed linear subspace of $E$ as in Corollary \ref{2.3} and set 
 $$\mathscr{I} := \left\{T \in \B(E): \ker T \supset D \right\}.$$ 
It is clear that $\mathscr{I}$ is a left ideal of $\B(E)$, and it is weak*-closed since 
$$\mathscr{I}= \left\{ x \otimes \lambda : x \in D, \ \lambda \in E' \right\}^\perp.$$
 We shall show that this ideal fails to be weak*-topologically finitely-generated. Assume towards a contradiction that there exist $n \in \N$ and $T_1, \ldots, T_n \in \B(E)$ such that 
$$\mathscr{I} = \overline{\B(E)T_1+ \cdots + \B(E)T_n}^{w^*}.$$
We show that
\begin{equation*}		\label{eq2.1}
\bigcap_{T \in \mathscr{I}} \ker T = \bigcap_{i=1}^n \ker T_i.
\end{equation*}
Indeed, let $x \in \bigcap_{i=1}^n \ker T_i$, and $S \in \mathscr{I}$. Take a net $(S_\alpha)$ in $\B(E)T_1+ \cdots + \B(E)T_n$ converging to $S$ in the weak*-topology. Then for any $\lambda \in E'$ we have 
$$ 0 = \lim_\alpha \langle  S_\alpha(x), \lambda \rangle  = \lim_\alpha \langle x \otimes \lambda, S_\alpha \rangle = \langle x \otimes \lambda, S \rangle = \langle S(x), \lambda \rangle,$$
and as $\lambda$ was arbitrary it follows that $S(x) = 0$. As $x$ was arbitrary $ \bigcap_{i=1}^n \ker T_i \subset \bigcap_{T \in \mathscr{I}} \ker T$, and the reverse inclusion is trivial.

Observe that $D \subset \bigcap_{T \in \mathscr{I}} \ker T$. Conversely, given $x \in E \setminus D$, we may pick $\lambda \in E'$ such that $\langle x, \lambda \rangle =1$, and $\ker \lambda \supset D$. Then the operator 
$\vert x \rangle \langle \lambda \vert$ belongs to $\mathscr{I}$, but $\vert x \rangle \langle \lambda \vert(x) = x \neq 0$,
so that in fact $D = \bigcap_{T \in \mathscr{I}} \ker T= \bigcap_{i=1}^n \ker T_i$. However this contradicts the choice of $D$.
\end{proof}

This is the only example that we know of a dual Banach algebra which is not weak*-topologically left Noetherian. It would be interesting to know if there are examples of the form $M(G)$, the measure algebra of a locally compact group $G$. In the light of \cite[Corollary~1.6(i)]{Wh} this would be particularly interesting for a compact group $G$. It would also be interesting to know whether the Fourier--Stieltjes algebra of a locally compact group ever fails to be weak*-topologically left Noetherian.

Another interesting problem would be to characterise those closed linear subspaces $F$ of a non-separable Banach space $E$ such that $F$ is the kernel of some bounded linear operator on $E$.

\subsection*{Acknowledgements}
\noindent
The second author is supported by the French
``Investissements d'A\-v\-e\-n\-ir'' program, project ISITE-BFC (contract
ANR-15-IDEX-03). The article is based on part of the second author's PhD thesis, and as such we would like to thank Garth Dales, as well as the thesis examiners Gordon Blower and Tom K\"orner, for their careful reading of earlier versions of this material and their helpful comments. We are grateful to Hugh Wark and Tomasz Kania for some useful email exchanges.

\end{document}